\documentclass[12pt]{amsart}
\usepackage{amsmath,amssymb,amsbsy,amsfonts,amsthm,latexsym,mathabx,
            amsopn,amstext,amsxtra,euscript,amscd,stmaryrd,mathrsfs,
            cite,array,mathtools,enumerate}
            
\usepackage{url}
\usepackage[colorlinks,linkcolor=blue,anchorcolor=blue,citecolor=blue,backref=page]{hyperref}

\usepackage{color}

\renewcommand*{\backref}[1]{}
\renewcommand*{\backrefalt}[4]{%
    \ifcase #1 (Not cited.)%
    \or        (p.\,#2)%
    \else      (pp.\,#2)%
    \fi}
\begin{document}

\newtheorem{theorem}{Theorem}
\newtheorem{lemma}[theorem]{Lemma}
\newtheorem{claim}[theorem]{Claim}
\newtheorem{cor}[theorem]{Corollary}
\newtheorem{prop}[theorem]{Proposition}
\newtheorem{definition}{Definition}
\newtheorem{question}[theorem]{Open Question}
\newtheorem{conj}[theorem]{Conjecture}
\newtheorem{prob}{Problem}
\newtheorem{algorithm}[theorem]{Algorithm}

\numberwithin{equation}{section}
\numberwithin{theorem}{section}

\def\squareforqed{\hbox{\rlap{$\sqcap$}$\sqcup$}}
\def\qed{\ifmmode\squareforqed\else{\unskip\nobreak\hfil
\penalty50\hskip1em\nobreak\hfil\squareforqed
\parfillskip=0pt\finalhyphendemerits=0\endgraf}\fi}

\def\cA{{\mathcal A}}
\def\cB{{\mathcal B}}
\def\cC{{\mathcal C}}
\def\cD{{\mathcal D}}
\def\cE{{\mathcal E}}
\def\cF{{\mathcal F}}
\def\cG{{\mathcal G}}
\def\cH{{\mathcal H}}
\def\cI{{\mathcal I}}
\def\cJ{{\mathcal J}}
\def\cK{{\mathcal K}}
\def\cL{{\mathcal L}}
\def\cM{{\mathcal M}}
\def\cN{{\mathcal N}}
\def\cO{{\mathcal O}}
\def\cP{{\mathcal P}}
\def\cQ{{\mathcal Q}}
\def\cR{{\mathcal R}}
\def\cS{{\mathcal S}}
\def\cT{{\mathcal T}}
\def\cU{{\mathcal U}}
\def\cV{{\mathcal V}}
\def\cW{{\mathcal W}}
\def\cX{{\mathcal X}}
\def\cY{{\mathcal Y}}
\def\cZ{{\mathcal Z}}

\def\fA{{\mathfrak A}}
\def\fJ{{\mathfrak J}}

\def\fV{{\mathfrak A}}

\def\sssum{\mathop{\sum\!\sum\!\sum}}
\def\ssum{\mathop{\sum\ldots \sum}}

\def\Xm{\cX_m}

\def \C {{\mathbb C}}
\def \F {{\mathbb F}}
\def \L {{\mathbb L}}
\def \K {{\mathbb K}}
\def \N {{\mathbb N}}
\def \R {{\mathbb R}}
\def \Q {{\mathbb Q}}
\def \Z {{\mathbb Z}}

\def \bW {{\mathbf W}}

\def\barG{\overline{\cG}}
\def\\{\cr}
\def\({\left(}
\def\){\right)}
\def\fl#1{\left\lfloor#1\right\rfloor}
\def\rf#1{\left\lceil#1\right\rceil}

\newcommand{\pfrac}[2]{{\left(\frac{#1}{#2}\right)}}

\def\rem{\mathrm{\, rem~}}

\def \Prob{{\mathrm {}}}
\def\e{\mathbf{e}}
\def\ep{{\mathbf{\,e}}_p}
\def\epp{{\mathbf{\,e}}_{p^2}}
\def\er{{\mathbf{\,e}}_r}
\def\eM{{\mathbf{\,e}}_M}
\def\eps{\varepsilon}
\def\Res{\mathrm{Res}}
\def\vec#1{\mathbf{#1}}

\def \li {\mathrm {li}\,}

\def\ip{\overline p}
\def\ipd{\ip_d}
\def\iq{\overline q}

\def\e{{\mathbf{\,e}}}
\def\ep{{\mathbf{\,e}}_p}
\def\em{{\mathbf{\,e}}_m}

\def\mand{\qquad\mbox{and}\qquad}

\newcommand{\commB}[1]{\marginpar{%
\begin{color}{red}
\vskip-\baselineskip 
\raggedright\footnotesize
\itshape\hrule \smallskip B: #1\par\smallskip\hrule\end{color}}}

\newcommand{\commI}[1]{\marginpar{%
\begin{color}{blue}
\vskip-\baselineskip 
\raggedright\footnotesize
\itshape\hrule \smallskip I: #1\par\smallskip\hrule\end{color}}}

\title[Large sieve inequality for sparse sequences]{\bf On the exponential large sieve inequality for sparse sequences modulo primes}

\author[M. C. Chang]{Mei-Chu Chang}

\address{Department of Mathematics, University of California.
Riverside,  CA 92521, USA}
\email{mcc@math.ucr.edu}

 \author[B. Kerr] {Bryce Kerr}

\address{Department of Pure Mathematics, University of New South Wales,
Sydney, NSW 2052, Australia}
\email{bryce.kerr89@gmail.com}

 \author[I. E. Shparlinski] {Igor E. Shparlinski}

\address{Department of Pure Mathematics, University of New South Wales, 
Sydney, NSW 2052, Australia}
\email{igor.shparlinski@unsw.edu.au}

\subjclass[2010]{11L07, 11N36}
 \keywords{exponential sums, sparse sequences, large sieve}

\begin{abstract} We complement the argument of M.~Z.~Garaev (2009) with 
several other ideas to obtain a stronger  version of the large sieve 
inequality with  sparse exponential sequences of the form $\lambda^{s_n}$. In particular, we obtain a result which is
 non-trivial  for monotonically increasing sequences $\cS=\{s_n \}_{n=1}^{\infty}$ provided $s_n\le n^{2+o(1)}$, whereas the original  argument of M.~Z.~Garaev requires $s_n \le n^{15/14 +o(1)}$ in the same setting. 
 We also give an application of our result to arithmetic properties of integers with almost 
 all digits prescribed.
 \end{abstract}
\maketitle


\section{Introduction}
\label{sec:intro} 

The classical  large sieve inequality, giving upper  bounds on average values of various 
trigonometric and Dirichlet polynomials with essentially arbitrary sequences $\cS=\{s_n\}_{n=1}^T$ ,
has proved to be an extremely useful  and versatile tool in analytic number theory and harmonic 
analysis, see, for example,~\cite{IwKow,Kow,Ram}. Garaev and Shparlinski~\cite[Theorem~3.1]{GarShp} 
have introduced a modification of the large sieve, for both trigonometric and Dirichlet polynomials with arguments that 
contain exponentials of $\cS$ rather than 
the elements of $\cS$. In the case of trigonometric polynomials, 
Garaev~\cite{Gar}  has introduced a new approach,  which has led to a stronger version of 
the  the exponential large sieve inequality, improving 
some of the results of~\cite{GarShp}, see also~\cite[Lemma 2.11]{BGLS} and~\cite[Theorem 1]{Shp2} for 
several other bounds of this type. Furthermore, stronger versions of  the exponential large sieve inequality
for special sequences $\cS$, such as $T$ consecutive integers or the first $T$ primes, can also be found in~\cite{BGLS,GarShp}, with some 
applications given in~\cite{Shp1}.  

Here we continue this direction and concentrate on the case of general sequences $\cS$ without any
arithmetic restriction.  We introduce  several new ideas which allow us
to improve some results of Garaev~\cite{Gar}. For example, we make 
use of the bound of~\cite[Theorem~5.5]{KS} on exponential sums over small 
multiplicative subgroups modulo $p$, which hold for almost all primes $p$, see Lemma~\ref{lem:aa}.
We also make the method more flexible so it now applies to much sparser sequences $\cS$ than in~\cite{Gar}. 

More precisely, let us fix some integer $\lambda\ge 2$. For each prime number $p$, we let $t_p$ denote the order of $\lambda \bmod{p}$.
For  real $X$ and $\Delta$ we define the set
$$
\cE_\Delta(X)=\{ p\le X~:~t_p\ge \Delta \}.
$$
Note that by a result of  Erd\"{o}s and Murty~\cite{ErMu}, see also~\eqref{eq: E X}, 
for $\Delta=X^{1/2}$ almost all primes $p \le X$ belong to $\cE_\Delta(X)$. 

For integer $T$ and two sequences 
of complex weights  $\Gamma =\{\gamma_n\}_{n=1}^T$  and
integers $\cS=\{s_n\}_{n=1}^T$ 
we define the sums 
$$
V_\lambda(\Gamma, \cS;T,X,\Delta)=
\sum_{p \in \cE_\Delta(X)} \max_{\gcd(a,p)=1}\left|\sum_{n\le T}\gamma_n\e_p(a\lambda^{s_n})\right|^{2}, 
$$
where $\e_r(z) = \exp(2 \pi i z/r)$.

These sums  majorize the ones considered by Garaev~\cite{Gar} where each 
term is divided by the divisor function $\tau(p-1)$ of $p-1$ . Here we obtain  a new bound of the sums 
$V_\lambda(\Gamma, \cS;T,X,\Delta)$ which in particular improves some bounds of Garaev~\cite{Gar}.

The argument of Garaev~\cite{Gar} reduces the problem to bounding Gauss sums for which he uses the bound of Heath-Brown and Konyagin~\cite{HbKg}, 
that is, the admissible pair~\eqref{eq:HBK1}, which is defined below. In particular, 
for $V_\lambda(\Gamma, \cS;T,X,X^{1/2})$ the result of Garaev~\cite{Gar} is nontrivial provided 
\begin{equation}
\label{eq:GarThresh}
S\le X^{15/14+o(1)}.
\end{equation}
Our results by-pass significantly  the threshold~\eqref{eq:GarThresh}
allow to replace $15/14$ with any fixed $\vartheta < 2$. 
 
%

Our improvement is based on a modification of the argument of Garaev~\cite{Gar} which 
allows us to use the bounds of short sums with exponential functions, given in~\cite[Theorem~5.5]{KS},
see also Lemma~\ref{lem:aa} below. This alone allows us to extend the result of~\cite{Gar}
to sparse sequences $\cS$, roughly growing as  at most $s_n \le  n^{7/6-\varepsilon}$ for 
any fixed  $\varepsilon > 0$ in the same scenario where the result of~\cite{Gar} limits the growth 
to $s_n \le  n^{15/14-\varepsilon}$.  Furthermore, using bounds of exponential sums over small subgroups of finite fields, in particular that of 
Bourgain,    Glibichuk and  Konyagin~\cite{BGK}  we relax the condition on $\cS$ to  
$s_n \le  n^{3/2-\varepsilon}$. 

Using a different argument which combines a bound of Bourgain and Chang~\cite{BC} for Gauss sums modulo a product of two primes with a duality principle for bilinear forms, we  obtain another, although less explicit bound which 
allows the elements to grow as fast as $s_n \le  n^{2-\varepsilon}$. Furthermore, for this
result we do not need to limit the summation to primes from $\cE_\Delta(X)$ but can 
consider all primes from $p\le X$, in which case we denote 
$$
V_\lambda(\Gamma, \cS;T,X)=
\sum_{p \in X} \max_{\gcd(a,p)=1}\left|\sum_{n\le T}\gamma_n\e_p(a\lambda^{s_n})\right|^{2}.
$$


We also give an application of our new estimate to investigating arithmetic properties of 
integers with almost  all digits prescribed in some fixed base. To simplify the exposition, 
we only  consider binary expansions (and hence we talk about bits rather than binary digits). 
Namely, for an integer $S \ge 1$, an $S$-bit integer $a$ and  a sequence 
of integers $\cS=\{s_n\}_{n=1}^T$  with $0 \le s_1< \ldots < s_T\le S$, we denote by 
$\cN(a; \cS)$ the set of $S$-bit integers $z$ whose bits on all positions $j=1, \ldots, S$ 
(counted from the right) must agree with those of $a$ except maybe when $j \in \cS$.

We first recall that Bourgain~\cite{Bour1,Bour2}
has recently obtained 
several very strong results about the distribution of prime numbers among the 
elements of $\cN(a; \cS)$, see also~\cite{HaKa}. However,  in the setting of 
the strongest result in this direction from~\cite{Bour2}, the set $\cS$ of 
``free'' positions has to be very massive, namely its cardinality  has to satisfy $T\ge (1-\kappa)S$
for some small (and unspecified) absolute constant $\kappa>0$. 
In the case of square-free numbers instead of prime numbers, a similar result  has been 
obtained in~\cite{DES} with any fixed $\kappa < 2/5$ (one can also find in~\cite{DES} 
some results on the distribution of the value of the Euler function and quadratic non-residues 
in $\cN(a; \cS)$).  Here we  address some problems at the other extreme, and relax the strength of 
arithmetic  conditions 
on the elements from $\cN(a; \cS)$  but instead  consider much sparse sets $\cS$  of available positions. 

\section{Main results}

Throughout the paper, the letter $p$ always denotes a prime number. 

As usual $A= O(B)$,  $A \ll B$, $B \gg A$ are all equivalent to $|A| \le c |B|$ for some 
 {\it absolute\/}  constant $c> 0$, whereas $A=o(B)$ means that $A/B\to 0$.

We say that a pair $(\alpha, \beta)$ is {\it admissible\/} if for any prime $p$ and 
any integer $\lambda$ with $\gcd(\lambda,p)=1$ we have
$$\max_{(a,p)=1}\left|\sum_{z=1}^{t}e_p(a \lambda^{z})\right|\le t^{\alpha}p^{\beta+o(1)},$$
as $p\to \infty$,  where $t$ is the multiplicative order of $\lambda$ modulo $p$.

Concerning admissible pairs, Korobov~\cite{Kor} has shown that the pair 
$$(\alpha, \beta) =(0,1/2),$$ is admissible. For shorter ranges of $t,$ Korobov's bound has been improved by Heath-Brown and Konyagin~\cite{HbKg} who show that the pairs
\begin{equation}
\label{eq:HBK1}
(\alpha, \beta) = (5/8,1/8),
\end{equation}
and 
\begin{equation}
\label{eq:HBK2}
(\alpha, \beta) =   (3/8, 1/4), 
\end{equation}
are admissible.

 More recently Shkredov~\cite{Shk1,Shk2} has shown that the pair 
\begin{equation}
\label{eq:Shk}
(\alpha, \beta) = (1/2,1/6),
\end{equation}
is admissible, which improves on the pairs~\eqref{eq:HBK1} and~\eqref{eq:HBK2} in the medium 
range of $t$.  

Furthermore, the truly remarkable result of 
Bourgain,    Glibichuk and  Konyagin~\cite{BGK} implies that for any $\zeta>0$ 
there is some $\vartheta> 0$  that depends only on $\zeta$ such that 
\begin{equation}
\label{eq:BGK}
(1-\vartheta,   \zeta \vartheta  ), 
\end{equation}
is admissible.

Our first result is as follows.

\begin{theorem}
\label{thm:main1}
Suppose that for an admissible pair $(\alpha, \beta)$ and
some positive numbers $\eta$ and $\delta$, we have
\begin{equation}
\label{eq:abconditions}
\frac{\beta +\eta}{1-\alpha}\le \frac{1}{2}-\delta.
\end{equation}
Suppose further that $S$, $T$ and $X$ 
are parameters satisfying 
\begin{equation}
\label{eq:TSX}
T^{1+1/(3-2\alpha)}\ge SX^{2\eta}.
\end{equation}
Let $\Delta>1$ and integer $k\ge 1$ satisfy
\begin{equation}
\label{eq:kassumption}
X\le \left(\left(\frac{T}{SX^{2\eta}} \right)^{1/(3-2\alpha)}\Delta\right)^{k}.
\end{equation}
Then for any sequence of  complex numbers 
 $\Gamma =\{\gamma_n\}_{n=1}^T$ with $|\gamma_n|\le 1$  and
integers $\cS=\{s_n\}_{n=1}^T$  with $0 \le s_1< \ldots < s_T\le S$ 
we have
\begin{align*}
V_\lambda(\Gamma, \cS&;T,X,\Delta) \\
&\le \left(X+
TX^{-\delta/(k^2+2)} + \(S^{2-2\alpha}TX^{-2\eta}\)^{1/(3-2\alpha)}\right)TX^{1+o(1)}.
\end{align*}
\end{theorem}

We note that under~\eqref{eq:TSX} the condition~\eqref{eq:kassumption} also 
follows from a simpler inequality
$$
X\le \(T^{-1/(3-2\alpha)^2}\Delta\)^{k}.
$$
Considering the strength of Theorem~\ref{thm:main1}, we take $\Delta=X^{1/2}$ and $T=X^{1+\varepsilon}$. Using  the admissible pair of Heath-Brown and Konyagin~\eqref{eq:HBK1}, we obtain a power saving in Theorem~\ref{thm:main1} provided $S\le X^{7/6-\varepsilon}$, improving of Garaev's range of $S\le X^{15/14-\varepsilon}$. With the same choice of parameters and using the admissible pair of Bourgain, Glibichuk and Konyagin~\eqref{eq:BGK}, we obtain a power saving in Theorem~\ref{thm:main1} provided $S\le X^{3/2-\varepsilon}$.

Using a different method we can set $\Delta =1$ and also 
 extend the range of $S$ for which we may obtain a nontrivial bound 
 for $V_\lambda(\Gamma, \cS;T,X)$ at the cost of making the power saving explicit.
 
\begin{theorem}
\label{thm:main2}
There exists some absolute constant  $\rho>0$ such that 
$$
V_\lambda(\Gamma, \cS;T,X) \le \left(X^{1-\rho}T^2+X^{3/2}T^{3/2}+X^{3/4}T^{7/8}S^{1/4}\right)X^{o(1)}.
$$
\end{theorem}

Comparing the bound of Theorem~\ref{thm:main2} with the trivial bound $XT^2$, we see that 
it is nontrivial provided
$$T>X^{1+\varepsilon} \quad \text{and} \quad S<TX^{1+\varepsilon},$$
which on taking $T=X^{1+\varepsilon}$, we obtain a power saving in Theorem~\ref{thm:main2} provided 
$S\le T^{2-\varepsilon}$. 


For a sequence of points $\cA=\{a_n\}_{n=1}^{T}$ we define the discrepancy $D$ of $\cA$ by 
$$D=\sup_{0\le a \le b \le 1}\left|\frac{A(a,b)}{T}-(b-a)\right|,$$ 
where $A(a,b)$ denotes the number of points of $\cA$ falling in the interval $[a,b] \in [0,1]$.
 Garaev~\cite{Gar} combines his bound for  $V_{\lambda}(\Gamma,\cS,T,X,\Delta)$ 
 with a result of Erd\"{o}s and Murty~\cite{ErMu},
 \commI{Expanded with~\eqref{eq: E X}, split one sentence into 2 } 
 which in particular implies that 
\begin{equation}
\label{eq: E X}
\cE_{X^{1/2},X} = (1+o(1)) \frac{X}{\log X}, \qquad X \to \infty, 
\end{equation}
 and the Erd\"{o}s-Tur\'{a}n inequality (see for example~\cite{DrTi}). 
 This allows Garaev~\cite[Section~3]{Gar}  to show that for any $\varepsilon > 0$ there is some $\delta > 0$ such 
that for almost all primes $p\le X$, the sequence 
\begin{equation}
\label{eq:Aseq}
A(\lambda,p)=\left\{\frac{\lambda^{s_n}}{p} \bmod{1} \right\}_{1\le n \le T},
\end{equation}
with $T= \rf{X(\log{X})^{2+\varepsilon}}$, has discrepancy
$$
D\le (\log{T})^{-\delta},
$$
provided $S\le X^{15/14+o(1)}$ as $X\to \infty$.

 For comparison with our bound, Theorem~\ref{thm:main2} produces the following result. For any 
 $\varepsilon>0$ and almost all primes  $p\le X$, the sequence~\eqref{eq:Aseq} with 
$T=\rf{X^{1+\varepsilon}}$ has discrepancy
$$D\le T^{-\delta},$$
provided $S\le X^{2-\varepsilon}$ as $X\rightarrow \infty$.

We now give an application of  Theorem~\ref{thm:main2} to the numbers with prescribed digits, 
namely to the integers from the set $\cN(a; \cS)$, defined in Section~\ref{sec:intro}. 
We denote by  $\omega(k)$ the number of distinct prime divisors of an integer $k\ge 1$.

 \begin{theorem}
\label{thm:main3} Let us fix some $\varepsilon> 0$. 
For any sequence of  
integers $\cS=\{s_n\}_{n=1}^T$  with $0 \le s_1< \ldots < s_T\le S$ with 
$$
S \le T^{2 - \varepsilon},
$$
and  any $S$-bit integer $a$,  we have
$$
\omega\(\prod_{z \in \cN(a; \cS)} z\) \gg  T^{1+\delta}
$$
for some $\delta > 0$ which depends only on $\varepsilon$. 
\end{theorem}

\section{Preliminary results}

We recall that  $A \ll B$ and $A = O(B)$ are both equivalent to the inequality 
$|A| \leq c B$ for some constant $c$, which throughout the paper may depend on $q$ 
and occasionally, where obvious, on the  integer parameter  $k\ge 1$.

We alslo use $\Sigma^*$ to indicate that the summation is taken over a reduced 
residue system. That is, for any function $\psi$ and integer $k$, we have 
$$
\sideset{}{^*} \sum_{c \bmod k} \psi(c) 
= \sum_{\substack{c=1\\ \gcd(c,k)=1}}^k \psi(c). 
$$

%
%
%

We need the following simplified form of the large sieve inequality,  
see~\cite[Theorem~7.11]{IwKow}.

\begin{lemma}
\label{lem:largesieve}
For any $K\ge 1$ and increasing sequence of integers $\cS=\{s_n\}_{n=1}^T$  with $\max_{s\in \cS}s = S$, 
we have
$$
\sum_{k\le K}\ \sideset{}{^*} \sum_{c \bmod k} \left|\sum_{n\le T}\gamma_n\e_k(cs_n) \right|^2 \ll (K^2+S)T.
$$
\end{lemma}

The following is~\cite[Theorem~5.5]{KS}.

\begin{lemma}
\label{lem:aa}
For each integer $t$ and prime $\ell \equiv 1 \bmod t$ we fix some element $g_{t,\ell}$ of multiplicative order $t$ modulo $\ell$. Then, for any fixed integer $k\ge 2$ and an arbitrary $U>1$, the bound
$$\max_{(a,\ell)=1}\left|\sum_{x=0}^{t-1}\e_{\ell}(ag_{t,\ell}^x)\right|\ll t \ell^{1/2k^2}(t^{-1/k}+U^{-1/k^2}),$$
holds for all primes $\ell \equiv 1 \bmod t$ except at most  $U/\log{U}$ of them.
\end{lemma}

\begin{lemma}
\label{lem:ordp}
Let $\lambda$ be a fixed integer. For any $Z>0$ we have
$$\#\{p~\text{prime}~:~\mathrm{ord}_p\,{\lambda}\le Z \ \}\ll Z^2.$$
\end{lemma}
\begin{proof}
If $\mathrm{ord}_p\,{\lambda}=y$ then $\lambda^{y}-1\equiv 0 \bmod {p}$. This implies that
$$
\#\{p~\text{prime}~:~\mathrm{ord}_p\,{\lambda}<Z \} \le \omega\(\prod_{1\le z \le Z}\(\lambda^{z}-1\)\),
$$
where as before,  $\omega(k)$ denotes the number of distinct prime divisors of an integer $k\ge 1$.
Hence, 
$$
\#\{p~\text{prime}~:~\mathrm{ord}_p\,{\lambda}<Z \ \} \ll \log{\prod_{1\le z \le Z}(\lambda^z-1)}
\le \log \(\lambda^{Z^2/2}\)\ll Z^2,
$$
which gives the desired result. 
\end{proof}

The following is a special case of~\cite[Corollary~4.2]{BC}.
\begin{lemma}
\label{lem:subgroupcomp}
Let $p_1$ and $p_2$ be primes and let $\cH$ be a subgroup of $\Z_{q}^{*}$, where $q = p_1p_2$ 
such that 
$$\# \{ \cH \bmod{p_\nu}\}\ge  q^{\delta}, \qquad \nu = 1,2$$
for some fixed $\delta > 0$.
Then
$$
\max_{\gcd(a,q)=1}\left|\sum_{h\in \cH}\e_q(ah) \right|\le (\#\cH)^{1-\varrho},
$$
for some $\varrho>0$ which depends only on $\delta > 0$.
\end{lemma}

\section{Proof of Theorem~\ref{thm:main1}}


\subsection{Initial tranformations}
Let
$$
\sigma_{p}(a)=\sum_{n\le T}\gamma_n\e_p(a\lambda^{s_n}).
$$
It is also convenient to define $a_p$ as any integer $a\in \{1, \ldots, p-1\}$
with 
\begin{equation}
\label{eq:ap}
\left|\sigma_{p}(a_p)\right| = \max_{\gcd(a,p)=1}\left|\sigma_{p}(a)\right|,
\end{equation}
so that
$$
V_\lambda(\Gamma, \cS;T,X,\Delta)=
\sum_{\substack{p\in \cE_\Delta(X)}} \left|\sigma_{p}(a_p)\right|^{2} .
$$
However, it is more convenient to work with the sums  where each 
term is divided by the divisor function $\tau(p-1)$. We define
$$
W_\lambda(\Gamma, \cS;T,X,\Delta)=
\sum_{\substack{p \in \cE_\Delta(X)}} \frac{1}{\tau(p-1)} \left|\sigma_{p}(a)\right|^{2},
$$
and note the inequality $\tau(n)=n^{o(1)}$ implies that 
$$
V_\lambda(\Gamma, \cS;T,X,\Delta)\le W_\lambda(\Gamma, \cS;T,X,\Delta)X^{o(1)}.
$$
Hence it is enough to prove
\begin{equation}
\begin{split}
\label{eq:thm1bound}
W_\lambda(\Gamma&, \cS;T,X,\Delta)\\
& \le \left(X+\frac{S^{1-1/(3-2\alpha)}T^{1/(3-2\alpha)}}{X^{2\eta/(3-2\alpha)}}+\frac{T}{X^{\delta/(k^2+2)}} \right)TX^{1+o(1)},
\end{split}
\end{equation}
where $\alpha,\beta,\delta,\eta$ satisfy~\eqref{eq:abconditions} and $(\alpha,\beta)$ is an admissible pair.

Fix some $p\le X$ and consider $\sigma_p(a_p)$. We split $s_n$ into arithmetic progressions mod $t_p$.
Using the orthogonality of exponential functions, we obtain 
\begin{align*}
\sigma_p(a_p) &=\sum_{x=1}^{t_p}\sum_{\substack{n\le T \\ s_n \equiv x \bmod t_p}}\gamma_n \e_p(a_p \lambda^{s_n})\\
& = \frac{1}{t_p}\sum_{x=1}^{t_p}\sum_{b=1}^{t_p}\sum_{n\le T}\gamma_n \e_{t_p}(b(s_n-x))\e_p(a_p \lambda^x),
\end{align*}
and hence
\begin{align*}
\sigma_p(a_p)&=\frac{1}{t_p}\sum_{d \mid t_p}\sum_{x=1}^{t_p}
\sum_{\substack{b=1 \\ \gcd(b,t_p)=d}}^{t_p}
\sum_{n\le T}\gamma_n \e_{t_p}(b(s_n-x))\e_p(a_p\lambda^{x}) \\ &=\frac{1}{t_p}\sum_{d \mid t_p}\sum_{x=1}^{t_p}\ \sideset{}{^*} \sum_{b \bmod {(t_p/d)}} \,\sum_{n\le T}\gamma_n \e_{t_p/d}(b(s_n-x))\e_p(a_p\lambda^{x}).
\end{align*}
Let $\xi> 0$ be a  real parameter to be chosen later.
We set 
$$
D_p = \xi t_p,
$$
and
 partition summation over $d$ according to $D_p$. This gives 
\begin{equation}
\label{eq:partition1}
|\sigma_p(a_p)|\le |\sigma_{p,1}(a_p)|+|\sigma_{p,2}(a_p)|,
\end{equation}
where
\begin{equation}
\label{eq:sigma1p}
\sigma_{p,1}(a_p)=\frac{1}{t_p}\sum_{\substack{d \mid t_p \\ d\le D_p}}\sum_{x=1}^{t_p}\ \sideset{}{^*} \sum_{b \bmod {(t_p/d)}} \,\sum_{n\le T}\gamma_n \e_{t_p/d}(b(s_n-x))\e_p(a_p\lambda^{x}),
\end{equation}
and
$$
\sigma_{p,2}(a_p)=\frac{1}{t_p}\sum_{\substack{d \mid t_p \\ d> D_p}}\sum_{x=1}^{t_p}
\ \sideset{}{^*} \sum_{b \bmod {(t_p/d)}} \,\sum_{n\le T}\gamma_n \e_{t_p/d}(b(s_n-x))\e_p(a_p\lambda^{x}).
$$
The equation~\eqref{eq:partition1} implies that 
$$
|\sigma_p(a_p)|^2\ll |\sigma_{p,1}(a_p)|^2+|\sigma_{p,2}(a_p)|^2,
$$
which on averaging over $p\le X$ gives
\begin{equation}
\label{eq:sigma12p}
W_\lambda(\Gamma, \cS;T,X,\Delta)  \ll \Sigma_1+\Sigma_2,
\end{equation}
where 
\begin{equation}
\begin{split}
\label{eq:sigma12}
\Sigma_1 &=\sum_{p\in \cE_\Delta(X)}\frac{1}{\tau(p-1)} |\sigma_{p,1}(a_p)|^2,\\
\Sigma_2  &=\sum_{p\in \cE_\Delta(X)}\frac{1}{\tau(p-1)} |\sigma_{p,2}(a_p)|^2.
\end{split}
\end{equation}

\subsection{The sum  $\Sigma_1$}
To bound $\Sigma_1$ we  use the argument of Garaev~\cite[Theorem~3.1]{Gar}. Fix some $p\le X$ and 
consider $\sigma_{p,1}(a_p)$.  From~\eqref{eq:sigma1p} and the Cauchy-Schwarz inequality 
\begin{align*}
|\sigma_{p,1}(a_p)|^2&=\left|\frac{1}{t_p}\sum_{\substack{d \mid t_p \\ d\le D_p}}\sum_{x=1}^{t_p}
\ \sideset{}{^*} \sum_{b \bmod {(t_p/d)}} \,
\sum_{n\le T}\gamma_n \e_{t_p/d}(b(s_n-x))\e_p(a_p\lambda^{x})\right|^2 \\
&\le \frac{\tau(t_p)}{t_p}\sum_{\substack{d \mid t_p \\ d\le D_p}}\sum_{x=1}^{t_p}
\left|\ \sideset{}{^*} \sum_{b \bmod {(t_p/d)}} \,\sum_{n\le T}\gamma_n \e_{t_p/d}(b(s_n-x)) \right|^2.
\end{align*}
Expanding the square and interchanging summation gives
\begin{align*}
|\sigma_{p,1}(a_p)|^2  &  \le \frac{\tau(t_p)}{t_p}\sum_{\substack{d \mid t_p \\ d\le D_p}}
\ \sideset{}{^*} \sum_{b_1,b_2\bmod {(t_p/d)}} \\
& \qquad \quad \sum_{n_1,n_2\le T}\gamma_{n_1}\overline\gamma_{n_2}\e_{t_p/d}(b_1s_{n_1}-b_2s_{n_2})\sum_{x=1}^{t_p}\e_{t_p/d}(x(b_2-b_1)). 
\end{align*}
By the orthogonality of exponential functions, the inner sum vanishes unless $b_1 = b_2$. Hence 
\begin{align*}
|\sigma_{p,1}(a_p)|^2 & \le \tau(t_p)\sum_{\substack{d \mid t_p \\ d\le D_p}}\ \sideset{}{^*} \sum_{b \bmod {(t_p/d)}} \,\sum_{n_1,n_2\le T}\gamma_{n_1}\overline\gamma_{n_2}\e_{t_p/d}(b(s_{n_1}-s_{n_2})) \\
&\le  \tau(p-1)\sum_{\substack{d \mid t_p \\ d\le D_p}}\ \sideset{}{^*} \sum_{b \bmod {(t_p/d)}} \,\left|\sum_{n\le T}\gamma_n \e_{t_p/d}(bs_n) \right|^2,
\end{align*}
where we have used the inequality $$\tau(t_p)\le \tau(p-1),$$ since $t_p\mid (p-1)$.
Summing over $p\le X$ we see that 
$$
\Sigma_1\le \sum_{p\le X}\sum_{\substack{d \mid t_p \\ d\le D_p}}\ \sideset{}{^*} \sum_{b \bmod {(t_p/d)}} \,\left|\sum_{n\le T}\gamma_n \e_{t_p/d}(bs_n) \right|^2.
$$
We define the sequence of numbers $X_j$ for $1\le j \le J$, 
where 
\begin{equation}
\label{eq:def J}
J = \rf{\frac{\log (X/\Delta)}{\log 2}},
\end{equation}
by 
\begin{equation}
\label{eq:def Xj}
X_1=\Delta, \qquad X_j=\min\{2X_{j-1},X\}, \ 2 \le j \le J,
\end{equation}
and partition the set of primes $p \le X$ into the sets
\begin{equation}
\label{eq:def Rj}
\cR_j=\{p\le X~:~X_j\le t_p<X_{j+1}  \}.
\end{equation}
Writing
$$\Sigma_{1,j}= \sum_{p\in \cR_j}\sum_{\substack{d \mid t_p \\ d\le D_p}}\ \sideset{}{^*}
 \sum_{b \bmod {(t_p/d)}} \,\left|\sum_{n\le T}\gamma_n \e_{t_p/d}(bs_n) \right|^2,$$
we have 
\begin{equation}
\label{eq:Sigma1j}
\Sigma_1 \ll \sum_{j=1}^J\Sigma_{1,j}.
\end{equation}
For each integer $r$, we define the set $\cQ(r)$ by
\begin{equation}
\label{eq:Set Qr}
\cQ(r)=\{ p \le X~:~t_p=r\},
\end{equation}
so that, replacing $t_p$ with $r$ for $p \in \cQ(r)$, we obtain
\begin{align*}
\Sigma_{1,j}&\le \sum_{X_j\le r <2X_j}\sum_{p\in \cQ(r)}\sum_{\substack{d \mid r \\ d\le D_p}}
\ \sideset{}{^*} \sum_{b \bmod {(r/d)}} \,\left|\sum_{n\le T}\gamma_n \e_{r/d}(bs_n) \right|^2 \\
&= \sum_{X_j\le r <2X_j}\#\cQ(r)\sum_{\substack{d \mid r \\ d\le D_p}}
\ \sideset{}{^*} \sum_{b \bmod {(r/d)}} \,\left|\sum_{n\le T}\gamma_n \e_{r/d}(bs_n) \right|^2.
\end{align*} 
For each prime $p\in \cQ(r)$ we have $r\mid (p-1)$ and hence for $X_j\le r <2X_j$ we also have
$$\#\cQ(r)\le \frac{X}{r}\le \frac{X}{X_j} \mand D_p < 2 \xi X_j.$$
This implies that 
\begin{align*}
\Sigma_{1,j}&\le \frac{X}{X_j}\sum_{X_j\le r<2X_j}\sum_{\substack{d \mid r \\ d\le 2 \xi X_j}}
\ \sideset{}{^*} \sum_{b \bmod {(r/d)}} \,
\left|\sum_{n\le T}\gamma_n\e_{r/d}(bs_n)\right|^2 \\
&=\frac{X}{X_j}\sum_{d\le 2 \xi X_j}\sum_{\substack{X_j\le r<2X_j \\ d \mid r}}
\ \sideset{}{^*} \sum_{b \bmod {(r/d)}} \,\left|\sum_{n\le T}\gamma_n\e_{r/d}(bs_n)\right|^2,
\end{align*}
and hence
\begin{equation}
\label{eq:Fj}
\Sigma_{1,j}\le \frac{X}{X_j}\sum_{d \le 2 \xi X_j}F_j(d),
\end{equation}
where $F_j(d)$ is given by 
$$
F_j(d)=\sum_{\substack{X_j/d\le m<2X_j/d}} 
\ \sideset{}{^*} \sum_{b \bmod m} \,
\left|\sum_{n\le T}\gamma_n\e_{m}(bs_n)\right|^2.
$$
An application of  Lemma~\ref{lem:largesieve} gives
$$
F_j(d)\ll \left(\frac{X_j^2}{d^2}+S\right)T,
$$ 
which combined with~\eqref{eq:Fj} implies that 

$$
\Sigma_{1,j}\le \frac{X}{X_j}\sum_{d\le 2 \xi X_j}\left(\frac{X^2_j}{d^2}+S \right)T\ll \frac{X}{X_j}\left(X_j^2+2 \xi X_j S\right)T,
$$
and hence by~\eqref{eq:Sigma1j}
\begin{equation}
\label{eq:sigma1bound}
\Sigma_{1}\ll \sum_{j =1}^J  \frac{X}{X_j}\left(X_j^2+ \xi X_j S\right)T \ll X(X+ \xi S\log{X} )T.
\end{equation}

\subsection{The sum  $\Sigma_2$}
 Fix some $p\le X$ and consider $\sigma_{p,2}(a_p)$. For each value of $d$ in  the outermost summation we split summation over $x$ into arithmetic progressions mod $t_p/d$. Recalling that $\sigma_{p,2}(a_p)$ is given by
$$
\sigma_{p,2}(a_p)=\frac{1}{t_p}\sum_{\substack{d \mid t_p \\ d> D_p}}\sum_{x=1}^{t_p}
\ \sideset{}{^*} \sum_{b \bmod {(t_p/d)}} \,\sum_{n\le T}\gamma_n \e_{t_p/d}(b(s_n-x))
\e_p\(a_p\lambda^{x}\),
$$
 we see that
\begin{align*}
\sigma_{p,2}(a_p)=\frac{1}{t_p}\sum_{\substack{d \mid t_p \\ d>D_p}}&\sum_{y=1}^{t_p/d}\ \sideset{}{^*} \sum_{b \bmod {(t_p/d)}} \\
&\quad \sum_{n\le T}\gamma_n\e_{t_p/d}(b(s_n-y))\sum_{z=1}^{d}
\e_p\(a_p\lambda^{y}\lambda^{zt_p/d}\),
\end{align*}
and hence

\begin{align*}
|\sigma_{p,2}(a_p)|&\le \frac{1}{t_p}\sum_{\substack{d \mid t_p \\ d>D_p}}\sum_{y=1}^{t_p/d}\left|\ \sideset{}{^*} \sum_{b \bmod {(t_p/d)}} \,\sum_{n\le T}\gamma_n \e_{t_p/d}(b(s_n-y)) \right|\\
& \qquad \qquad \qquad \qquad \qquad \qquad \qquad \qquad \times \left|\sum_{z=1}^{d}
\e_p\(a_p\lambda^{y}\lambda^{zt_p/d}\) \right| \\
&\le \sum_{\substack{d \mid t_p \\ d>D_p}}\frac{1}{t_p}\sum_{y=1}^{t_p/d}\left|\ \sideset{}{^*} \sum_{b \bmod {(t_p/d)}} \,\sum_{n\le T}\gamma_n \e_{t_p/d}(b(s_n-y)) \right|\\
& \qquad \qquad \qquad \qquad \qquad \qquad \qquad \qquad \times 
 \left|\sum_{z=1}^{d}\e_p\(f_{d,p}\lambda^{zt_p/d}\) \right|,
\end{align*}
where $f_{d,p}$ is chosen to satisfy
$$
 \left|\sum_{z=1}^{d}\e_p\(f_{d,p}\lambda^{zt_p/d}\) \right|
=  \max_{\gcd(a,p)=1} \left|\sum_{z=1}^{d}\e_p\(a \lambda^{zt_p/d}\) \right|.
$$
Let 
$$
U(p,d)=\frac{1}{t_p}\sum_{y=1}^{t_p/d}\left|\sum_{\substack{b=1 \\ \gcd(b,t_p/d)=1}}^{t_p/d}\sum_{n\le T}\gamma_n \e_{t_p/d}(b(s_n-y)) \right|,
$$
so that 
\begin{equation}
\label{eq:Sigma and U}
|\sigma_{p,2}(a_p)|\le\sum_{\substack{d \mid t_p \\ d>D_p}}U(p,d)\left|\sum_{z=1}^{d}\e_p(f_{d,p}\lambda^{zt_p/d}) \right|.
\end{equation}
We consider bounding the terms $U(p,d)$. By the Cauchy-Schwarz inequality
\begin{align*}
U(p,d)^2&\le \frac{1}{dt_p}\sum_{y=1}^{t_p/d}\left|\sum_{\substack{b=1 \\ \gcd(b,t_p/d)=1}}^{t_p/d}\sum_{n\le T}\gamma_n \e_{t_p/d}(b(s_n-y)) \right|^2 \\
&=\frac{1}{dt_p}\sum_{1\le n_1,n_2 \le T}\sum_{\substack{b_1,b_2 = 1\\\gcd(b_1b_2,t_p/d)=1}}^{t_p/d}\gamma_{n_1}\overline\gamma_{n_2}\e_{t_p/d}(b_1s_{n_1}-b_2s_{n_2})\\
& \qquad \qquad \qquad \qquad \qquad \qquad  \qquad \qquad
\times \sum_{y=1}^{t_p/d}\e_{t_p/d}(y(b_1-b_2)). 
\end{align*}
Using the orthogonality of exponential functions again,  we see that the last sums vanishes 
unless $b_1 = b_2$. This gives
$$
U(p,d)^2 \le  \frac{1}{d^2}\sum_{1\le n_1,n_2 \le T}\sum_{\substack{b=1 \\ \gcd(b,t_p/d)=1}}^{t_p/d}\gamma_{n_1}\overline\gamma_{n_2}\e_{t_p/d}(b(s_{n_1}-s_{n_2})).
$$
After rearranging and extending the summation over $b$ to the complete residue system 
modulo $t_p/d$, we derive
\begin{align*}
U(p,d)^2 &\le  \frac{1}{d^2}\sum_{\substack{b=1 \\ \gcd(b,t_p/d)=1}}^{t_p/d}  \sum_{1\le n_1,n_2 \le T}\gamma_{n_1}\overline\gamma_{n_2}\e_{t_p/d}(b(s_{n_1}-s_{n_2}))\\
& = \frac{1}{d^2}\sum_{\substack{b=1 \\ \gcd(b,t_p/d)=1}}^{t_p/d}  \left| \sum_{1\le n \le T}\gamma_{n}\e_{t_p/d}(b s_{n})\right|^2\\
& \le \frac{1}{d^2} \sum_{b=1}^{t_p/d}  \left| \sum_{1\le n \le T}\gamma_{n}\e_{t_p/d}(b s_{n})\right|^2  =\frac{ t_p}{ d^3} V\left(t_p/d\right),
\end{align*}
where for an integer $r\ge 1$ we define
\begin{equation}
\label{eq:def Vt}
V(r)=\#\{(n_1,n_2)\in [1,T]^2~:~ s_{n_1}\equiv s_{n_2} \bmod{r}\}.
\end{equation}
Substituting this in~\eqref{eq:Sigma and U} gives
\begin{equation*}
|\sigma_{p,2}(a_p)|\le t_p^{1/2}\sum_{\substack{d \mid t_p \\ d>D_p}}\frac{1}{d^{3/2}}
V\left(t_p/d\right)^{1/2}\left|\sum_{z=1}^{d}\e_p\(f_{d,p}\lambda^{zt_p/d}\) \right|.
\end{equation*}
Summing over $p\le X$ gives
$$
\Sigma_2\le \sum_{p \in \cE_\Delta(X)}\frac{t_p}{\tau(p-1)} \left(\sum_{\substack{d \mid t_p \\ d>D_p}}\frac{1}{d^{3/2}}
V\left(t_p/d\right)^{1/2} \left|\sum_{z=1}^{d}\e_p\(f_{d,p}\lambda^{zt_p/d}\) \right| \right)^2,
$$
which by the Cauchy-Schwarz inequality implies that  
\begin{align*}
\Sigma_2&\le  \sum_{p \in \cE_\Delta(X)}\frac{t_p\tau(t_p)}{\tau(p-1)} \sum_{\substack{d \mid t_p \\ d>D_p}}\frac{1}{d^{3}}V\left(t_p/d\right) \left|\sum_{z=1}^{d}\e_p\(f_{d,p}\lambda^{zt_p/d}\) \right|^2 \\ 
&\le  \sum_{p \in \cE_\Delta(X)}t_p \sum_{\substack{d \mid t_p \\ d>D_p}}\frac{1}{d^{3}}V\left(t_p/d\right) \left|\sum_{z=1}^{d}\e_p\(f_{d,p}\lambda^{zt_p/d}\) \right|^2.
\end{align*}
At this point our strategy is to rearrange summation so we may apply Lemma~\ref{lem:aa}.  We define the sequence  $X_j$  as in~\eqref{eq:def Xj}, we let $\cQ(r)$ be given by~\eqref{eq:Set Qr} and for each integer $r$ we define the following subsets $S_{i}(r)$ of  $\cQ(r)$
$$
\cS_{i}(r)=\{p~:~2^{i}\le p \le 2^{i+1}  \text{ and } t_p=r \}. 
$$

Writing
$$
\Sigma_{2,i,j} =
\sum_{X_j\le r \le X_{j+1}} r \sum_{p \in \cS_{i}(r)} \sum_{\substack{d \mid r \\ d>D_p}}\frac{1}{d^{3}}V\left(r/d\right) \left|\sum_{z=1}^{d}\e_p\(f_{d,p}\lambda^{zr/d}\) \right|^2, 
$$
the above implies that
$$
\Sigma_2 \le  \sum_{i =1}^J\sum_{j : X_j \ll 2^{i}}\Sigma_{2,i,j}.
$$

To further transform the sums $\Sigma_{2,i,j},$ define the numbers $Z_{j}$ by 
\begin{equation}
\label{eq:Zij}
Z_{j}=\xi X_j, \qquad j =1, \ldots, J,
\end{equation}
so that
$$
\Sigma_{2,i,j} \ll X_{j}\sum_{X_j\le r \le X_{j+1}}\sum_{p \in \cS_{i}(r)}\sum_{\substack{d \mid r \\ d>Z_{j}}}\frac{1}{d^{3}}V\left(r/d\right) \left|\sum_{z=1}^{d}\e_p\(f_{d,p}\lambda^{zr/d}\) \right|^2.
$$
After interchanging summation, we arrive at 
\begin{equation}
\begin{split}
\label{eq:Sigma2ij}
\Sigma_{2,i,j} \ll X_{j}\sum_{Z_{j}<d \le X_{j+1}}\frac{1}{d^3}&
\sum_{\substack{X_j\le r \le X_{j+1} \\ d \mid r}}V\left(r/d\right)\\
&\qquad \sum_{p \in \cS_{i}(r)}
  \left|\sum_{z=1}^{d}\e_p\(f_{d,p}\lambda^{zr/d}\) \right|^2.
\end{split}
\end{equation}

Let $\rho$ be a parameter to be chosen later. 
We now  partition summation over $i$ and $j$ in $\Sigma_2$ as follows 
\begin{equation}
\label{eq:Sigma2ij123}
\Sigma_2 \le\Sigma_2^{\le}+\Sigma_2^{\ge},
\end{equation}
where 
$$
\Sigma_2^{\le}= \sum_{i =1}^J\sum_{j : X_j \le 2^{i \rho}}\Sigma_{2,i,j}
\mand 
\Sigma_2^{\ge}= \sum_{i =1}^J\sum_{j : 2^{i \rho}\le X_j \ll 2^{i}}\Sigma_{2,i,j}.
$$

To estimate $\Sigma_2^{\le}$, we first fix some $j$ with $X_j \le 2^{i\rho}$.
Considering the inner summation over $p$, we partition $\cS_{i}(r)$ according to Lemma~\ref{lem:aa}.
Let 
$$U_{i}(r)=\frac{2^{i(1-1/(k^2+2))}}{r^{1-2/(k^2+2)}},$$
and for integer $k$ we define the sets $S^{(1)}_{i}(r)$ and $S^{(2)}_{i}(r)$ by
\begin{align*}
\cS^{(1)}_{i}(r)&=\biggl\{ p\in \cS_{i}(r)~:\\
& \qquad \qquad \quad ~ \left|\sum_{z=1}^{d}\e_p\(f_{d,p}\lambda^{zr/d}\) \right|\le d 2^{i/2k^2}\(d^{-1/k}+U_{i}(r)^{-1/k^2}\)\biggr\},\\
\cS^{(2)}_{i}(r)&= \cS_{i}(r) \setminus \cS^{(1)}_{i}(r).
\end{align*}
Lemma~\ref{lem:aa} implies that 
$$\# \cS^{(2)}_{i}(r)\ll \frac{U_{i}(r)}{\log U_{i}(r)}.$$
Considering $\cS^{(1)}_{i}(r)$ and using the fact that $r \mid p-1$ for $p\in  \cS_{i}(r)$ gives
\begin{equation}
\label{eq:bound Sir}
\# \cS^{(1)}_{i}(r)\le \# \cS_{i}(r) \ll \frac{2^{i}}{r},
\end{equation}
which implies that 
\begin{align*}
\sum_{p \in \cS_{i}(r)}& \left|\sum_{z=1}^{d}\e_p\(f_{d,p}\lambda^{zr/d}\) \right|^2\\
& \qquad  \ll 
d^2\left(\frac{  2^{i(1+1/k^2)}}{r}(d^{-2/k}+U_{i}(r)^{-2/k^2})+\frac{U_{i}(r)}{\log U_{i}(r)}\right).
\end{align*}
Recalling the choice of $U_{i}(r)$ we see that
\begin{equation*}
\sum_{p \in \cS_{i}(r)}\left|\sum_{z=1}^{d}\e_p\(f_{d,p}\lambda^{zr/d}\) \right|^2\ll \frac{d^22^{i(1-1/(k^2+2))}}{r^{1-2/(k^2+2)}}+\frac{d^{2-2/k}2^{i(1+1/k^2)}}{r},
\end{equation*}
which on assuming that 
\begin{equation}
\label{eq:DeltaXassumption}
X\le \left(\xi \Delta \right)^{k},
\end{equation}
simplifies to
\begin{equation}
\label{eq:Snontrivial}
\sum_{p \in \cS_{i}(r)}\left|\sum_{z=1}^{d}\e_p\(f_{d,p}\lambda^{zr/d}\) \right|^2\ll \frac{d^22^{i(1-1/(k^2+2))}}{r^{1-2/(k^2+2)}}.
\end{equation}
Hence considering  $\Sigma_{2,i,j}$, we have 
\begin{align*}
\Sigma_{2,i,j}&\ll X_{j}2^{i(1-1/(k^2+2))}\sum_{Z_{j}<d \le X_{j+1}}\frac{1}{d}\sum_{\substack{X_j\le r \le X_{j+1} \\ d \mid r}}\frac{V\left(r/d\right)}{r^{1-2/(k^2+2)}} \\
&\ll X_{j}2^{i(1-1/(k^2+2))}\sum_{Z_{j}<d \le X_{j+1}}\frac{1}{d^{2-2/(k^2+2)}}\sum_{\substack{X_j/d\le  r \le X_{j+1}/d }}\frac{V(r)}{r^{1-2/(k^2+2)}},
\end{align*}
after the change of variable $r \to dr$. Writing
$$W_j(d)= \sum_{\substack{X_j/d\le  r \le X_{j+1}/d }}\frac{V(r)}{r^{1-2/(k^2+2)}},$$
the above implies
\begin{equation}
\label{eq:Sigma2W}
\Sigma_{2,i,j}\ll  X_{j}2^{i(1-1/(k^2+2))}\sum_{Z_{j}<d \le X_{j+1}}\frac{W_j(d)}{d^{2-2/(k^2+2)}}.
\end{equation}

Considering the sum $W_j(d)$ and recalling the definition of $V(r)$ given by~\eqref{eq:def Vt}, 
we have
\begin{align*}
W_j(d) & =\sum_{\substack{X_j/d\le  r \le X_{j+1}/d }}\sum_{\substack{1\le n_1,n_2 \le T \\ s_{n_1} \equiv s_{n_2} \bmod {r} }}\frac{1}{r^{1-2/(k^2+2)}} \\ &
\ll \left(\frac{d}{X_j}\right)^{1-2/(k^2+2)}\sum_{1\le n_1,n_2 \le T}\sum_{\substack{X_j/d\le  r \le X_{j+1}/d \\ s_{n_1} \equiv s_{n_2} \bmod {r}}} 1.
\end{align*}
Considering the last sum on the right, we have
\begin{align*}
\sum_{1\le n_1,n_2 \le T}\sum_{\substack{X_j/d\le  r \le X_{j+1}/d \\ s_{n_1} \equiv s_{n_2} \bmod {r}}} 1\ll  \frac{TX^{j}}{d}+\sum_{\substack{1\le n_1<n_2\le T}}\sum_{\substack{X_j/d\le  r \le X_{j+1}/d \\ s_{n_1} \equiv s_{n_2} \bmod {r}}} 1.
\end{align*}
Since the term
$$\sum_{\substack{X_j/d\le  r \le X_{j+1}/d \\ s_{n_1} \equiv s_{n_2} \bmod {r}}} 1,$$
is bounded by the number of divisors of $s_{n_2}-s_{n_1}$, we see that 
$$
\sum_{\substack{X_j/d\le  r \le X_{j+1}/d \\ s_{n_1} \equiv s_{n_2} \bmod {r}}} 1= S^{o(1)},
$$
and hence 
\begin{equation}
\label{eq:sum div}
\sum_{1\le n_1,n_2 \le T}\sum_{\substack{X_j/d\le  r \le X_{j+1}/d \\ s_{n_1} \equiv s_{n_2} \bmod {r}}} 1\ll  \left(\frac{X_j}{d}+TS^{o(1)}\right)T,
\end{equation}
which gives
$$
W_j(d)\le \left(\frac{d}{X_j}\right)^{1-2/(k^2+2)}\left(\frac{X_j}{d}+TS^{o(1)}\right)T.
$$
Substituting the above into~\eqref{eq:Sigma2W} we get 
\begin{align*}
\Sigma_{2,i,j}&\ll  X_{j}^{1+2/(k^2+2)}2^{i(1-1/(k^2+2))}T\sum_{Z_{j}<d \le X_{j+1}}\frac{1}{d^{2}} \\
& + X_{j}^{2/(k^2+2)}2^{i(1-1/(k^2+2))}T^2S^{o(1)}\sum_{Z_{j}<d \le X_{j+1}}\frac{1}{d},
\end{align*}
which simplifies to
\begin{align*}
\Sigma_{2,i,j} & \le \frac{X_{j}^{1+2/(k^2+2)}2^{i(1-1/(k^2+2))}T}{Z_{j}}+X_{j}^{2/(k^2+2)}2^{i(1-1/(k^2+2))}T^2(SX)^{o(1)} \\ 
&\le X_{j}^{2/(k^2+2)}2^{i(1-1/(k^2+2))}\left(\frac{1}{\xi}+T\right)T(SX)^{o(1)},
\end{align*}
on recalling the choice of $Z_{j}$ given by~\eqref{eq:Zij}. 

We now  assume that 
\begin{equation}
\label{eq:etaassumption}
\xi \ge \frac{1}{T}.
\end{equation}
Without loss of generality, we can also assume that $S= X^{O(1)}$ and thus 
$(SX)^{o(1)} = X^{o(1)}$.
Hence, the above bounds further simplify to
$$
\Sigma_{2,i,j}  \le T^2 X_{j}^{2/(k^2+2)}2^{i(1-1/(k^2+2))}X^{o(1)}.
$$
 Summing over $i$ and $j$ with $X_j \le 2^{i \rho}$ we arrive at
$$
\Sigma_2^{\le} \le T^2X^{o(1)}\sum_{i=1}^{J}\sum_{j: X_j \le 2^{i \rho }}X_{j}^{2/(k^2+2)}2^{i(1-1/(k^2+2))},
$$
and hence 
\begin{equation}
\label{eq:sigma<bound}
\Sigma_2^{\le}\le  T^2 X^{1-(1-2\rho)/(k^2+2)}X^{o(1)}.
\end{equation}

We next consider $\Sigma_2^{\ge}$. 
We begin our treatment of $\Sigma_2^{\ge}$ in a similar fashion to $\Sigma_2^{\le}$. 
In particular, we use~\eqref{eq:Sigma2ij} and the assumption that $(\alpha, \beta)$ is admissible to 
obtain
\begin{equation}
\label{eq:Sigma2ij-ab}
\Sigma_{2,i,j}  \le 2^{i(2\beta +o(1))}X_{j}\sum_{X_j\le r \le X_{j+1}}\# \cS_{i,j}(r)\sum_{\substack{d \mid r \\ d>Z_{j}}}\frac{1}{d^{3-2\alpha}}V\left(r/d\right),
\end{equation}
as $i \to \infty$. 

Using~\eqref{eq:bound Sir} and then  rearranging the order of summation,  the above reduces to
\begin{align*}
\Sigma_{2,i,j}  
& \le 2^{i(1+2\beta +o(1))}\sum_{X_j\le r \le X_{j+1}}\sum_{\substack{d \mid r \\ d>Z_{j}}}\frac{1}{d^{3-2\alpha}}V\left(r/d\right)\\
& \le 2^{i(1+2\beta +o(1))}\sum_{Z_j<d\le X_{j+1}}\frac{1}{d^{3-2\alpha}}W_j(d),
\end{align*}
where 
$$
W_j(d)=\sum_{X_j/d \le r \le X_{j+1}/d}V(r).
$$
We see from the definition~\eqref{eq:def Vt} that
$$
W_j(d)=\sum_{1\le n_1,n_2 \le T}\sum_{\substack{X_j/d\le r \le X_{j+1}/d \\ s_{n_1}\equiv s_{n_2} \bmod {r}}}1 \le \left(\frac{X_j}{d}+TS^{o(1)} \right)T,
$$
and hence
\begin{align*}
\Sigma_{2,i,j}& \le 2^{i(1+2\beta+o(1))}T\left(X_j\sum_{Z_j<d \le X_{j+1}}\frac{1}{d^{4-2a}}+TS^{o(1)}\sum_{Z_j \le d \le X_{j+1}}\frac{1}{d^{3-2\alpha}}\right) \\
&\le 2^{i(1+2\beta+o(1))}T\left(\frac{X_j}{Z_j^{3-2\alpha}}+\frac{TS^{o(1)}}{Z_j^{2-2a}}\right).
\end{align*}
Since obviously $S \le X^{O(1)}$, we can replace both $2^{o(i)}$ and $S^{o(1)}$ 
with $X^{o(1)}$. 
Recalling the choice of $Z_j$  and  the assumption~\eqref{eq:etaassumption},  we get 
$$
\Sigma_{2,i,j}\ll \frac{2^{i(1+2\beta)}T}{\xi^{2(1-\alpha)}X_j^{2(1-\alpha)}}
\(\xi^{-1}+T  \) X^{o(1)} \le
 \frac{2^{i(1+2\beta)}T^2}{\xi^{2(1-\alpha)}X_j^{2(1-\alpha)}} X^{o(1)}  .
$$
This implies that 
\begin{equation}
\begin{split}
\label{eq:sigma>bound}
\Sigma_2^{\ge} & \le  \frac{1}{\xi^{2(1-\alpha)}}T^2X^{o(1)}   \sum_{i=1}^{J}\sum_{ j: 2^{i \rho}\le X_j \le 2^{i}}\frac{2^{i(1+2\beta)}}{X_j^{2(1-\alpha)}}\\
&\le  T^2 \frac{X^{1+2(\beta+\eta-\rho (1-\alpha))}}{\xi^{2(1-\alpha)}} X^{o(1)}.
\end{split}
\end{equation}
Substituting the bounds~\eqref{eq:sigma<bound} and~\eqref{eq:sigma>bound} 
 in~\eqref{eq:Sigma2ij123}, we see that 
\begin{equation}
\label{eq:sigma2bound}
\Sigma_2\le \left(\frac{1}{X^{(1-2\rho)/(k^2+2)}}+\frac{X^{2(\beta-\rho (1-\alpha))}}{\xi^{2(1-\alpha)}} \right)T^2X^{1+o(1)}.
\end{equation}

\subsection{Concluding the proof}
Substituting~\eqref{eq:sigma1bound} and~\eqref{eq:sigma2bound} in~\eqref{eq:sigma12p}, gives
\begin{align*}
 W_\lambda(\Gamma&, \cS;T,X,\Delta)\\
& \le \left(X+\xi S+\frac{T}{X^{(1-2\rho)/(k^2+2)}}+\frac{TX^{2(\beta-\rho (1-\alpha))}}{\xi^{2(1-\alpha)}} \right)TX^{1+o(1)}.
\end{align*}
Let $\eta>0$ be a parameter and make the substitution
$$
\rho=\frac{\beta+\eta}{1-\alpha}.
$$
The above transforms into
\begin{align*}
W_\lambda&(\Gamma, \cS;T,X,\Delta) \\
&\le \left(X+\xi S+\frac{T}{X^{(1-2(\beta+\eta)/(1-\alpha))/(k^2+2)}}+\frac{T}{\xi^{2(1-\alpha)}X^{2\eta}} \right)TX^{1+o(1)}.
\end{align*}
Next we chooise 
$$\xi=\left(\frac{T}{SX^{2\eta}} \right)^{1/(3-2\alpha)},$$
to balance the second and fourth terms.
 This gives
\begin{align*}
W_\lambda&(\Gamma, \cS;T,X,\Delta)\\
& \le \left(X+\frac{S^{1-1/(3-2\alpha)}T^{1/(3-2\alpha)}}{X^{2\eta/(3-2\alpha)}}+\frac{T}{X^{(1-2(\beta+\eta)/(1-\alpha))/(k^2+2)}} \right)TX^{1+o(1)}.
\end{align*}
We now note that the assumption~\eqref{eq:abconditions}
implies that 
\begin{align*}
W_\lambda(\Gamma, \cS&;T,X,\Delta)  \\ 
& \le \left(X+\(S^{2-2\alpha}TX^{-2\eta}\)^{1/(3-2\alpha)}+\frac{T}{X^{\delta/(k^2+2)}} \right)TX^{1+o(1)}.
\end{align*}
which is the desired bound.

Finally,  to complete the proof,  it remains to note that~\eqref{eq:DeltaXassumption} is satisfied by the assumption~\eqref{eq:kassumption} 
and~\eqref{eq:etaassumption} is satisfied by~\eqref{eq:TSX}.

\section{Proof of Theorem~\ref{thm:main2}}

\subsection{Initial tranformations}
As before, for each prime $p$ we define the number $a_p$ by~\eqref{eq:ap}.
Taking $Z=X^{1/4}$ in Lemma~\ref{lem:ordp} and recalling that $t_p$ denotes the order of $\lambda$ mod $p$, we have
\begin{equation}
\begin{split}
\label{eq:121}
V_\lambda\(\Gamma, \cS;T,X\) & \le  X^{1/2}T^2 + V_\lambda(\Gamma, \cS;T,X,X^{1/4}) \\
& = X^{1/2}T^2 +\sum_{p \in \cE_{X^{1/4}}(X)}\left|\sigma_p(a_p)\right|^2.
\end{split}
\end{equation}

We define the sequence of numbers $X_j$,  as in~\eqref{eq:def Xj}
with $\Delta =X^{1/4}$. We also define the sets $\cR_j$ as in~\eqref{eq:def Rj}
for $j =1, \ldots, J$ with $J$ given by~\eqref{eq:def J}. 

Hence, partitioning summation over $p$ in~\eqref{eq:121} according to $\cR_j$ gives,
$$
V_\lambda\(\Gamma, \cS;T,X\) \ll X^{1/2}T^2+\sum_{j=1}^J W_j, 
$$
where
\begin{align*}
W_j=\sum_{p\in \cR_j}\left|\sigma_p(a_p)\right|^2.
\end{align*}

We define the number $Y$ by
\begin{equation}
\label{eq:Ydef}
Y=\frac{X^{3/4}S^{1/4}}{T^{1/4}},
\end{equation}
and let $I$ be the largest integer $j$ with $X_j \le Y$ (since $S\ge T$ we obviously have 
$Y \ge X^{3/4}> X^{1/4}$ so $I$ is correctly defined). 

We  now further partition  the summation over $j$ and re-write~\eqref{eq:121} as
\begin{equation}
\label{eq:Wj12}
V_\lambda\(\Gamma, \cS;T,X\)\le X^{1/2}T^2+W^{\le}+W^{\ge},
\end{equation}
where
\begin{equation}
\label{eq:W12def}
W^{\le}=\sum_{j =1}^IW_j  \mand W^{\ge}=\sum_{j =I+1}^JW_j.
\end{equation}

\subsection{The sum $W^{\le}$}  
We fix some $j$ with $X^{1/4}\le X_j< Y$. Considering $W_j$, we define the sets
\begin{equation}
\label{eq:Vrdef}
\cV_j(r)=\{p\in \cR_j ~:~t_p=r\},
\end{equation}
so that 
\begin{equation}
\label{eq:WjUr}
W_j=\sum_{X_j<r\le 2X_j}U_{j,r},
\end{equation}
where $U_{j,r}$ is given by 
\begin{align*}
U_{j,r}=\sum_{p\in \cV_j(r)}\left|\sigma_p(a_p)\right|^2.
\end{align*}
For each $p\in \cV_j(r)$ we   define the complex number $c_{j,r,p}$ by
\begin{align*}
c_{j,r,p}=\frac{\overline \sigma_p(a_p)}{\left(\sum_{p\in \cV_j(r)}|\sigma_p(a_p)|^2\right)^{1/2}},
\end{align*}
so that
\begin{equation}
\label{eq:cpnorm}
\sum_{p\in \cV_j(r)}\left|c_{j,r,p} \right|^2=1,
\end{equation}
and writing
\begin{align*}
U_{j,r}^{*}=\sum_{p\in \cV_j(r)}\sum_{1\le n \le T}c_{j,r,p}\gamma_n \e_p(a_p\lambda^{s_n}),
\end{align*}
we see that 
\begin{equation}
\label{eq:Udual}
|U_{j,r}^{*}|=U_{j,r}^{1/2}.
\end{equation}
 We have
\begin{align*}
U_{j,r}^{*}=\sum_{0\le x<r}\sum_{p\in  \cV_j(r)}b_r(x)c_{j,r,p}\e_p(a_p\lambda^{x}),
\end{align*}
where 
\begin{equation}
\label{eq:brdef}
b_r(x)=\sum_{\substack{1\le n \le T \\ s_n \equiv x \bmod r}}\gamma_n,
\end{equation}
and hence by the Cauchy-Schwarz inequality
\begin{align*}
|U_{j,r}^{*}|^{2}\le\sum_{0\le x <r}|b_{r}(x)|^2\sum_{0\le x <r}\left|\sum_{p\in \cV_j(r)}c_{j,r,p} \e_{p}(a_p \lambda^{x}) \right|^2.
\end{align*}
Expanding the square and interchanging summation gives
\begin{align*}
|U_{j,r}^{*}|^{2}&\le\sum_{0\le x <r}|b_{r}(x)|^2\sum_{p_1,p_2\in \cV_j(r)}|c_{p_1}||c_{p_2}|\left|\sum_{0\le x <r}\e_{p_1p_2}((a_{p_1}p_2-a_{p_2}p_1)\lambda^{x}) \right|, 
\end{align*}
which implies that
\begin{align*}
|U_{j,r}^{*}|^2 &\le\left(\sum_{0\le x <r}|b_{r}(x)|^2\sum_{p\in \cV_j(r)}|c_{j,r,p}|^2\right)r \\ & \quad \quad +\sum_{0\le x <r}|b_{r}(x)|^2\sum_{\substack{p_1,p_2\in \cV_j(r) \\ p_1 \neq p_2}}|c_{p_1}||c_{p_2}|\max_{(a,p_1p_2)=1}\left|\sum_{0\le x <r}\e_{p_1p_2}(a\lambda^{x}) \right|.
\end{align*}
Since 
$$t_{p_1}=t_{p_2}=r,$$
the set 
$$H=\{ \ \lambda^{x} \bmod{p_1p_2}~:~0\le x<r \ \},$$
is a subgroup of $\Z^{*}_{p_1p_2}$ and from the inequalities 
$$r\ge X^{1/4}>(p_1p_2)^{1/8},$$
we see that the conditions of Lemma~\ref{lem:subgroupcomp} are satisfied. An application of Lemma~\ref{lem:subgroupcomp} gives
$$
|U_{j,r}^{*}|^2 \le\left(\sum_{0\le x <r}|b_{r}(x)|^2\sum_{p\in \cV_j(r)}|c_{j,r,p}|^2\right)r+\sum_{0\le x <r}|b_{r}(x)|^2\left(\sum_{\substack{p\in \cV_j(r)}}|c_{j,r,p}|\right)^2r^{1-\varrho},
$$
which by the Cauchy-Schwarz inequality implies that
$$
 |U_{j,r}^{*}|^2\le\sum_{0\le x <r}|b_{r}(x)|^2 \sum_{p\in \cV_j(r)}|c_{j,r,p}|^2\left(r+|\cV_j(r)|r^{1-\varrho}\right),
$$
and hence by~\eqref{eq:cpnorm}
$$
 |U_{j,r}^{*}|^2\le\sum_{0\le x <r}|b_{r}(x)|^2\left(r+|\cV_j(r)|r^{1-\varrho}\right).
$$
Since
\begin{equation}
\label{eq:Vrin}
|\cV_j(r)|\le \frac{X}{r},
\end{equation}
we get 
\begin{equation}
\label{eq:Ur*}
 |U_{j,r}^{*}|^2\le\left(r+\frac{X}{r^{\varrho}}\right) \sum_{0\le x <r}|b_{r}(x)|^2.
\end{equation}
Recalling~\eqref{eq:brdef} and the assumption each $|\gamma_n|\le 1$, we see that
$$
\sum_{0\le x <r}|b_{r}(x)|^2
=\sum_{1\le n_1,n_2 \le T}\gamma_{n_1}\overline \gamma_{n_2}\sum_{\substack{0\le x <r \\ s_{n_1}\equiv x \bmod{r} \\ s_{n_2}\equiv x \bmod{r}}}1 = V(r),
$$
where $V(r)$ is defined by~\eqref{eq:def Vt}. 
By~\eqref{eq:Ur*} we have
\begin{align*}
 |U_{j,r}^{*}|^2\le V(r)\left(r+\frac{X}{r^{\varrho}}\right),
\end{align*}
and hence by~\eqref{eq:Udual}
\begin{align*}
|U_{j,r}|\le V(r)\left(r+\frac{X}{r^{\varrho}}\right).
\end{align*}
Combining the above with~\eqref{eq:WjUr} gives
\begin{equation}
\label{eq:Wj123}
W_j \le \sum_{X_j<r \le 2X_{j}}V(r) \left(X_j+\frac{X}{X_j^{\varrho}}\right).
\end{equation}

As in the proof of Theorem~\ref{thm:main1}, see~\eqref{eq:sum div}, we have 
\begin{align*}
\sum_{X_j<r \le 2X_{j}}V(r) &\ll X_{j}T+\sum_{\substack{1\le n_1,n_2 \le T \\ n_1 \neq n_2}}\sum_{\substack{X_j<r\le 2X_{j} \\ s_{n_1}\equiv s_{n_2} \bmod{r}}}1 \\
& \le (X_j+TS^{o(1)})T\le T^{2+o(1)},
\end{align*}
where we have used the assumption $S \le T^2$ and $T>X$ as otherwise Theorem~\ref{thm:main2} is trivial. Substituting the above into~\eqref{eq:Wj123} gives
\begin{align*}
W_j\le \left(X_j+\frac{X}{X_j^{\varrho}}\right)T^{2+o(1)},
\end{align*}
and hence by~\eqref{eq:W12def}
\begin{equation}
\label{eq:W1bound}
W^{\le}\le \left(Y+XX_1^{-\varrho}\right)T^{2+o(1)} \le \left(Y+X^{1-\varrho/4}\right)T^{2+o(1)}.
\end{equation}

\subsection{The sum $W^{\ge}$}

  We fix some $j$ with $Y\le X_j\le X$ and arrange $W_j$ as follows
\begin{align*}
W_j&=\sum_{p\in \cR_j}\left|\sigma_p(a_p)\right|^2 \le T\sum_{p\in \cR_j}\left|\sigma_p(a_p)\right|,
\end{align*}
and hence there exists some sequence of complex numbers $c_{j,p}$ with $|c_{j,p}|=1$ such that 
\begin{align*}
W_j\le T\sum_{p\in \cR_j}\sum_{1\le n \le T}c_{j,p}\gamma_n \e_p(a_p\lambda^{s_n}).
\end{align*}
An application of the Cauchy-Schwarz inequality gives
\begin{align*}
W_j^2\le T^3\sum_{1\le n \le T}\left|\sum_{p \in \cR_j}c_{j,p} \e_p(a_p\lambda^{s_n}) \right|^2.
\end{align*}
Since the sequence $s_n$ is increasing and bounded by $S$, we see that 
\begin{align*}
W_j^{2}&\le T^3\sum_{1\le s \le S}\left|\sum_{p \in \cR_j}c_{j,p} \e_p(a_p\lambda^{s}) \right|^2 \ll \frac{T^3}{S}\sum_{-S\le r,s\le S}\left|\sum_{p \in \cR_j}c_{j,p} \e_p(a_p\lambda^{r+s}) \right|^2,
\end{align*}
so that writing 
$$\bW_j=\sum_{-S\le r,s\le S}\left|\sum_{p \in \cR_j}c_{j,p} \e_p(a_p\lambda^{r+s}) \right|^2,$$
the above implies
\begin{equation}
\label{eq:WW'}
W_j^2\le \frac{T^3}{S}\bW_j.
\end{equation}
Considering $\bW_j$, expanding the square and interchanging summation gives
\begin{align*}
\bW_j&\le \sum_{p_1,p_2\in \cR_j}\left|\sum_{-S\le r,s\le S}\e_{p_1p_2}((a_{p_1}p_2-a_{p_2}p_1)\lambda^{r+s}) \right| \\
&\le S^2 |\cR_j|+\sum_{\substack{p_1,p_2\in \cR_j \\ p_1\neq p_2}}\sum_{-S\le r \le S}\left|\sum_{-S\le s \le S}\e_{p_1p_2}(a_{p_1p_2}\lambda^{r+s}) \right|,
\end{align*}
for some integers $a_{p_1p_2}$ with $\gcd(a_{p_1p_2},p_1p_2)=1$. By~\eqref{eq:Vrdef} and~\eqref{eq:Vrin}
$$
|\cR_j|=\sum_{X_j<r\le 2X_j}|\cV_j(r)|\ll X,
$$
and hence
\begin{equation}
\label{eq:WV}
\bW_j\ll S^2X +\sum_{\substack{p_1,p_2\in \cR_j \\ p_1 \neq p_2}}Z(p_1,p_2).\end{equation}
where 
$$
Z(p_1,p_2)= \sum_{-S\le r \le S}\left|\sum_{-S\le s \le S}\e_{p_1p_2}(a_{p_1p_2}\lambda^{r+s}) \right|.
$$

Considering $Z(p_1,p_2)$, by the Cauchy-Schwarz inequality, we have 
\begin{align*}
Z(p_1,p_2)^2&\ll S\sum_{-S\le r \le S}\left|\sum_{-S\le s \le S}\e_{p_1p_2}(a_{p_1p_2}\lambda^{r}\lambda^{s}) \right|^2 \\
&\ll S\left(1+\frac{S}{\text{ord}_{p_1p_2}(\lambda)} \right) \sum_{u \bmod{p_1p_2}}\left|\sum_{-S\le s \le S}\e_{p_1p_2}(a_{p_1p_2}u\lambda^{s}) \right|^2.
\end{align*}
Now, since 
\begin{align*}
\sum_{u \bmod{p_1p_2}}&\left|\sum_{-S\le s \le S}\e_{p_1p_2}(a_{p_1p_2}u\lambda^{s}) \right|^2\\
&=\sum_{-S\le s_1,s_2 \le S}\sum_{u \bmod{p_1p_2}}\e_{p_1p_2}(a_{p_1p_2}u(\lambda^{s_1}-\lambda^{s_2})) \\
&\ll p_1p_2S\left(1+\frac{S}{\text{ord}_{p_1p_1}(\lambda)} \right),
\end{align*}
we see that 
\begin{align*}
Z(p_1,p_2)^2&\ll p_1p_2S^{2}\left(1+\frac{S}{\text{ord}_{p_1p_1}(\lambda)} \right)^2\le X^2S^{2}\left(1+\frac{S}{\text{ord}_{p_1p_2}(\lambda)} \right)^2.
\end{align*}
Since $t_{p_1},t_{p_2}\ge X_j$, we have
\begin{align*}
\text{ord}_{p_1p_2}(\lambda)=\text{lcm}(t_{p_1},t_{p_2})=\frac{t_{p_1}t_{p_2}}{\gcd(t_{p_1},t_{p_2})}\ge \frac{X_j^2}{\gcd(p_1-1,p_2-1)},
\end{align*}
which implies
\begin{align*}
Z(p_1,p_2)^2\le X^2S^{2}\left(1+\frac{\gcd(p_1-1,p_2-1)S}{X_j^2} \right)^2,
\end{align*}
which after substituting the above in~\eqref{eq:WV} gives
\begin{align*}
\bW_j\ll S^2X+XS\sum_{\substack{p_1,p_2\in \cR_j \\ p_1 \neq p_2}}1+\frac{XS^2}{X_j^2}\sum_{\substack{p_1,p_2\in \cR_j \\ p_1 \neq p_2}}\gcd(p_1-1,p_2-1).
\end{align*}
We have 
\begin{align*}
\sum_{\substack{p_1,p_2\in \cR_j \\ p_1 \neq p_2}}1\le |\cR_j|^2\ll X^2,
\end{align*}
and 
\begin{align*}
\sum_{\substack{p_1,p_2\in \cR_j \\ p_1 \neq p_2}}\gcd(p_1-1,p_2-1)&\ll \sum_{1\le x_1<x_2 \le X}\gcd(x_1,x_2) \\ &=\sum_{1\le d \le X}d\sum_{\substack{1\le x_1<x_2\le X/d \\ (x_1,x_2)=1}}1 \ll X^{2+o(1)},
\end{align*}
so that 
$$
\bW_j\ll S^2X+SX^3+\frac{S^2X^{3+o(1)}}{X_j^2}.
$$
Combining the above with~\eqref{eq:WW'} gives
$$
W_j^2\ll SXT^3+X^3T^3+\frac{SX^{3+o(1)}T^3}{X_j^2},
$$
which simplifies to 
$$
W_j\le X^{3/2}T^{3/2}\left(1+\frac{S^{1/2}}{X_j} \right)X^{o(1)},
$$
since we may assume $S\le X^{2+o(1)}$.
By~\eqref{eq:W12def} we have
\begin{equation}
\label{eq:W2bound}
W^{\ge}\ll  X^{3/2}T^{3/2}\left(1+\frac{S^{1/2}}{Y} \right)X^{o(1)},
\end{equation}

\subsection{Concluding the proof}

Substituing~\eqref{eq:W1bound} and~\eqref{eq:W2bound} in~\eqref{eq:Wj12}
we derive
\begin{align*}
V_\lambda&\(\Gamma, \cS;T,X\) \\ & \quad\le X^{1/2}T^2+\left(Y+X^{1-\varrho}\right)T^{2+o(1)}+X^{3/2}T^{3/2}\left(1+\frac{S^{1/2}}{Y} \right)X^{o(1)}.
\end{align*}

Recalling the choice of $Y$ in~\eqref{eq:Ydef} the above simplifies to
\begin{align*}
V_\lambda\(\Gamma, \cS;T,X\)  \le \left(X^{1/2}T^2+X^{1-\varrho/4}T^2+X^{3/2}T^{3/2}+X^{3/4}T^{7/8}S^{1/4}\right)X^{o(1)},
\end{align*}
and the result follows with $\rho = \varrho/4$ (as clearly $\varrho \le 1$ and thus $\rho < 1/2$)

\section{Proof of Theorem~\ref{thm:main3}}

First we note that without loss of generality we may assume the binary digits of $a$ are zeros on all
positions $j \in \cS$.

For a prime $p$, let $N_p(a; \cS)$ be the number of $z \in \cN(a; \cS)$ with 
$p\mid z$. One can easily see that $N_p(a; \cS)$ is the number of solutions to the
congruence
$$
a + \sum_{n =1}^T d_n 2^{s_n} \equiv 0 \bmod p, \qquad d_n \in \{0,1\}, \ n =1, \ldots, T.
$$
We now proceed similarly to the proof of~\cite[Theorem~18.1]{KS}. Using the orthogonality of exponential functions, we write 
\begin{align*}
N_p(a; \cS) & =  \frac{1}{p}\sum_{b= 0}^{p-1} 
\sum_{(d_1, \ldots , d_T) \in \{0, 1\}^{T}}
\ep \(b\(\sum^{T}_{n=1}d_n 2^{s_n} + a \)\) \\
& = 2^{T}p^{-1}+
\frac{1}{p}\sum_{b= 1}^{p-1}
\sum_{(d_1, \ldots , d_T) \in \{0, 1\}^{T}}
\ep \(b\(\sum^{T}_{n=1}d_j 2^{s_n} + a \)\) \\
&  = 2^Tp^{-1}+
\frac{1}{p}\sum_{b= 1}^{p-1} \ep(ab)\prod_{n=1}^{T}
\(1 + \ep \(b  2^{s_n}\)\).
\end{align*}
Therefore,
\begin{equation}
\label{NnpP}
\left| N_{n,p}(a) - 2^Tp^{-1} \right|  \le Q_p,
\end{equation}
where 
$$
Q_p=  \max_{b= 1, \ldots, p-1}\left|  \prod_{n=1}^{T}
\(1 + \ep \(b  2^{s_n}\)\) \right|  .
$$
Using~\cite[Equation~(18.2)]{KS} we write 
\begin{equation}
\label{PSN}
Q_p \le \exp \( O(M_p \log  (T/M_p+1) )  \),
\end{equation}
where 
$$
M_p = \max_{\gcd(b,p)=1}\left|\sum_{n\le T}\ep(a\lambda^{s_n})\right|. 
$$
Now, 
by  Theorem~\ref{thm:main2}
 if we fix  some $\varepsilon_0> 0$, then there is  some  $\kappa>0$ such that  if 
 $$
 X = T^{1/(1+\varepsilon_0)}, \quad \Delta = X^{1/2} \mand S \le X^{2- \varepsilon_0},
 $$
then we have
$$
\sum_{\substack{p \in \cE_\Delta(X)}}M_p^{2} \le T^2 X^{1-\kappa}.
$$
Since $S  \le T^{2 - \varepsilon}$, to satisfy the above conditions, it is enough to define 
$\varepsilon_0$ by the equation 
$$
\frac{2- \varepsilon_0}{1+\varepsilon_0} = 2- \varepsilon
$$
or, more explicitely, 
$$
 \varepsilon_0  = \frac{\varepsilon}{3-\varepsilon} .
$$
Combining this with~\eqref{eq: E X}, we see that for all but $o(X/\log X)$ primes $p\le X$ 
we have $M_p \le T X^{-\kappa/3}$.  For each of these primes $p$, a combination 
of~\eqref{NnpP} and~\eqref{PSN} implies that $N_p(a; \cS)>0$ (provided that $p$ is 
large enough), which concludes the proof.

\section{Possible improvements}
\label{sec:Improve}

We note that one can get an improvement of Theorem~\ref{thm:main1} 
by using a combination of different admissible pairs depending on the range 
of $d$ in our treatement of the sum~\eqref{eq:Sigma2ij} in and thus making the choice of $\alpha$ and 
$\beta$ in~\eqref{eq:Sigma2ij-ab} dependent on $i$ and $j$. 

In particular, one can use the admissible pairs~\eqref{eq:HBK1}, \eqref{eq:HBK2}, \eqref{eq:Shk}  and~\eqref{eq:BGK} as well the admissible pairs of Konyagin~\cite{Kon} and Shteinikov~\cite{Sht} for small values of $d$ in~\eqref{eq:Sigma2ij}.

\section*{Acknowledgement}

This work was  partially supported   by  the NSF Grant DMS~1600154 (for M.-C.~C.) and
by ARC Grant~DP170100786 (for I.~S.).

\end{document}